\documentclass[12pt]{amsart}
\usepackage{latexsym, amsmath, amscd, amssymb, amsthm} 
\usepackage[english]{babel}
\newtheorem{theorem}{Theorem}
 \newtheorem{oz}{Definition}[section]

\begin{document}

\noindent

 \title[invariant polynomial transformations]{  A note about  invariant polynomial transformations of  integer sequences}

\author{Leonid Bedratyuk}\address{Department of Applied Mathematics,Khmelnitskiy national university, Insituts'ka, 11,  Khmelnitskiy, 29016, Ukraine}
\email{leonid.uk@gmail.com}
\begin{abstract}  We present an algorithm to find invariant poynomial transformations of integer sequences, using the classical invariant theory approach. 
\end{abstract}

\maketitle

\section{Introduction}

Let  $\mathcal{A}=(a_n)_{n \geq 0}$ be an integer sequence. A sequence   $\textbf{F}(\mathcal{A})=\left(b_n=f_n(a_0,a_1,\ldots,a_m)\right)_{n \geq 0}$ where  $f_n \in  \mathbb{Z}[x_0,x_1,\ldots,x_m],$ $m \geq n,$
 is called a  polynomial transformation of  the sequence $\mathcal{A}.$ In the sequel, only the polynomial transformations are considered. The composition  $\textbf{F} \circ \textbf{G}:= \textbf{F}(\textbf{G}(\mathcal{A}))$ of the two transformations $\textbf{F}$ and $\textbf{G}$ can be defined  in a natural way. A   transformation  $\textbf{G}$ is called the \textit{inverse} transformation of $\textbf{F}$, and it is denoted by $\textbf{F}^{-1},$  if for every sequence $\mathcal{A}$ we  have $\textbf{F}(\textbf{G}(\mathcal{A}))=\mathcal{A}.$
A transformation  $\textbf{F}$ is called  $\textbf{G}$-\textit{invariant}   if for every sequence $\mathcal{A}$ we  have $\textbf{F}(\textbf{G}(\mathcal{A}))=\textbf{F}(\mathcal{A})$.

For  instance, it is well known ( see e.g., Layman\,{\cite{Lay}; Spivey and  Steil \,\cite{SpS}}) that the Hankel  transformation  $\textbf{H}$ is $\textbf{B}_\mu$-invariant.    Here  $$\textbf{B}_\mu(\mathcal{A})=\left ( b_n=\sum_{i=0}^n {n \choose i} a_i \mu^{n-i} \mid \mu \in \mathbb{Q} \right)_{n \geq 0},$$ denotes the  $\mu$-binomial transformation  and  $\textbf{H}(\mathcal{A})=(h_n)_{n \geq 0},$ where   $h_n$ is the determinant of  Hankel matrix for  the elements  $a_0,a_1,\ldots, a_{2n}$:
 $$
h_n=\begin{vmatrix} a_0 &  a_1 & a_2 &\cdots  & a_n  \\
a_1 &  a_2 & a_3 &\cdots  & a_{n+1}  \\
\hdotsfor{5}\\
 a_{n-1} &  a_{n} & a_{n-1} &\cdots  & a_{2n-1}\\
 a_n &  a_{n+1} & a_{n+2} &\cdots  & a_{2n}  
\end{vmatrix}.
$$ 
This   determinant is well known in  classical invariant theory as the \textit{catalecticant }of a binary form, see \cite[p.232]{GrY}. The catalectiant  was introduced for the first time  by Sylvester in \cite{Sylv}. Also,   the transformation $\textbf{B}_\mu$  one may find in Hilbert's book  \cite[p.\,25]{Hilb}.  We can  prove that the Hankel transformation is $\textbf{B}_\mu$-invariant  by using the classical invariant theory approach.  In fact, let $\mathcal{D}$ be the following differential operator:
$$
\mathcal{D}=a_0 \partial_1+2 a_1 \partial_2+\cdots +2n a_{2n-1} \partial_{n} ,\mathcal{D}(a_0)=0, \partial_i:=\frac{\partial}{\partial a_i}.
$$
Put
$$
h'_n=\begin{vmatrix} b_0 &  b_1 & b_2 &\cdots  & b_n  \\
b_1 &  b_2 & b_3 &\cdots  & b_{n+1}  \\
\hdotsfor{5}\\
 b_{n-1} &  b_{n} & b_{n-1} &\cdots  & b_{2n-1}\\
 b_n &  b_{n+1} & b_{n+2} &\cdots  & b_{2n}  
\end{vmatrix}.
$$ 
Then Lie  \cite{Lie} implies
$$
h'_n=h_n+\mathcal{D}(h_n)\mu+\mathcal{D}^2(h_n)\frac{\mu^2}{2!}+\cdots +
\mathcal{D}^i(h_n)\frac{\mu^i}{i!}+\cdots
$$
By applying the determinant derivative rule we obtain after some calculation  that  $\mathcal{D}(h_n) = 0$ for all $n.$ 
Therefore  $h'_n=h_n.$ This condition is exactly equivalent to  the $\textbf{B}_\mu$-invariance of  Hankel transformation.

   This motivates us to consider  the following two general problems:   
  
 \textbf{Problem 1.} For  a fixed transformation $\textbf{F}$, find all $\textbf{F}$-invariant transformations. 
  
 \textbf{Problem 2.} For   a fixed transformation  $\textbf{F}$, find all transformations  $\textbf{G}$   such that   $\textbf{F}$ is  $\textbf{G}$-invariant transformation. 
  
  The aim of this paper is to develop an effective method for a  solution of the two above  problems for some special kinds of  transformations. The origin of  the method  came from the classical invariant theory and the theory of locally nilpotent derivations. We introduce the notion of exponential transformation and then prove  that for such transformations   Problem 1  can always be solved. 

In section 2 we give    a short introduction to the theory of locally nilpotent derivations and offer algorithms to solve  Problems 1 and 2.

In section 3  we  give  another   proof of  $\textbf{B}_\mu$-invariance  the Hankel  transformation  and  introduce several new  $\textbf{B}_\mu$-invariant transformations.  All of them came from the classical invariant theory. Also we  describe all $\textbf{B}_\mu$-invariant polynomial transformations in terms  of derivations.

In section 4  we illustrate the theory   by some examples.

\section{Derivations and automorphisms}

Let   $\varphi:\cup_{m \in \mathbb{N}}\mathbb{Z}[x_0,x_1,\ldots,x_m] \to \cup_{m \in \mathbb{N}}\mathbb{Z}[x_0,x_1,\ldots,x_m]$ be a  polynomial map. It means that  $\varphi$ is uniquely determined by the  set of polynomials  $\{ \varphi(x_n), \, n=0,1,\ldots \}.$ To any polynomial   map $\varphi$ and an  integer sequence $( a_n)_{n \geq 0}$ we  assign the  transformation $( \varphi(a_n))_{n \geq 0}.$   A  polynomial map $\varphi$ is said to be a \textit{polynomial automorphism} if there is a polynomial map $\psi$ such that $\varphi(\psi(x_n))=x_n$ for all $n.$ 

Denote by   $\mathbb{Z}[x_0,x_1,\ldots,x_m]^{\varphi}$ the \textit{algebra of   $\varphi$-invariants}: 
$$
\mathbb{Z}[x_0,x_1,\ldots,x_m]^{\varphi}:=\left\{ f \in  \mathbb{Z}[x_0,x_1,\ldots,x_m] \mid f(\varphi(x_0),\varphi(x_1),\ldots, \varphi(x_m))=f(x_0,x_1,\ldots, x_m) \right \}.
$$

The following theorem will be our   main computing tool in finding the invariant  polynomial transformations: 
\begin{theorem}
Let $\varphi$ be a polynomial map  and let   $\text{\textbf{F}}(\mathcal{A})=(\varphi(a_n))_{n \geq 0}$ be the corresponding integer transformation. Then the transformation   
$$\textbf{G}(\mathcal{A})= \left(g_n(a_0,a_1,\ldots,a_m) \right )_{n \geq 0},$$ is $\textbf{F}$-invariant if and only  if $g_n(x_0,a_1,\ldots,x_m) \in \mathbb{Z}[x_0,x_1,\ldots,x_m]^{\varphi}.$
\end{theorem}
The proof follows immediately  from the above definitions.

In general, the  problem of finding   the algebras of  $\varphi$-invariants  is  difficult. But in the case   when  $\varphi$  is so-called  exponential automorphism this  problem can be reduced to calculation of kernel of a derivation. 

A {\it derivation} of the algebra $\mathbb{Z}[x_0,x_1,\ldots,x_n]$ is a linear  map  $D$ satisfying the Leibniz rule: 
$$
D(f_1 \, f_2)=D(f_1) f_2+f_1 D(f_2), \text{  for all }  f_1, f_2 \in \mathbb{Z}[x_0,x_1,\ldots,x_n].
$$
A derivation $D$  is called {\it locally nilpotent} if for every $f \in \mathbb{Z}[x_0,x_1,\ldots,x_n]$ there is an $r \in \mathbb{N}$ such that $D^r(f)=0.$ 
The subalgebra 
$$
\ker D:=\left \{ f \in \mathbb{Z}[x_0,x_1,\ldots,x_n]|  D(f)=0 \right \},
$$
is called the {\it kernel} of the derivation $D.$

Any derivation   $D$ is completely determined by the elements $D(x_i).$ A  derivation   $D$  is called  \textit{linear} if  $D(x_i)$ is a linear form. A  linear locally nilpotent derivation is called a \textit{Weitzenb\"ock derivation}.
The  Weitzenb\"ock derivation defined by  $\mathcal{D}(x_0)=0, \mathcal{D}(x_i)=i x_{i-1}$ is called the \textit{basic Weitzenb\"ock derivation}. 
 There exists an isomorphism between the kernel $\ker \mathcal{D}$  and the algebra of covariants of a binary form,  a major object of research in the classical invariant theory
of the 19th century.  Here are a few examples of covariants: the discriminant, the resultant, the jacobian, the hessian, the catalectiant and the transvectant. The following theorem gives a  description of the algebra $\ker \mathcal{D}$  for a fixed number of involved variables:  
   
\begin{theorem} The kernel of the basic Weitzenb\"ock derivation $\mathcal{D}$ of $\mathbb{Q}[x_0,x_1,\ldots,x_n]$ is finitely generated algebra and 
$$
\ker \mathcal{D}=\mathbb{Q}[z_2,z_3,\ldots,z_n][x_0,x_0^{-1}] \cap \mathbb{Q}[x_0,x_1,\ldots,x_n],
$$
where  
$$
z_k=\sum_{i=0}^{k-2} (-1)^i {k \choose i} x_{k-i} x_1^i x_0^{k-i-1}+(k-1) (-1)^{k+1} x_1^k. 
$$
\end{theorem}
It is a classical result due to Cayley, see \cite{Gle}, page 164. The modern proof one may find   in \cite{Now}, \cite{Essen}.

How to find the kernel of arbitrary linear locally nilpotent derivation $D$?  Let us consider  the vector space (over $\mathbb{Q}$)  $X_n=< x_0,x_1,\ldots,x_n >.$  Suppose that there exists an  isomorphism  $\Psi:X_n   \to X_n  $  such that  $\Psi \mathcal{D}=  D \Psi.$ It implies   that  $\ker D =\Psi \left( \ker \mathcal{D} \right), $ i.e.,
$$
\ker D=\mathbb{Q}[\Psi(z_2),\Psi(z_3),\ldots,\Psi(z_n)][\Psi(x_0),\Psi(x_0)^{-1}] \cap \mathbb{Q}[x_0,x_1,\ldots,x_n].
$$
 Such an  isomorphism $\Psi$ is called a \textit{$(\mathcal{D},D)$-intertwining} isomorphism. 
Therefore, to describe the kernel of arbitrary Weitzenb\"ok derivation $D$ it is enough to know  the explicit  form of any   \textit{$(\mathcal{D},D)$-intertwining} isomorphism. 

An automorphism  $\varphi$ is called  \textit{exponential} if there exists a  locally nilpotent derivation $D$ such that 
$$
\varphi=\exp(D)=D^0+D+\frac{1}{2!}D^2+\cdots.
$$
For instance,    any automorphism of the form $$\varphi(x_n)=x_n+f(x_0,x_1,\ldots,x_{n-1}), f \in \mathbb{Z}[x_0,x_1,\ldots,x_{n-1}],$$ is exponential, see  Drensky an Yu \cite{Dren}. For any  exponential automorphism $\varphi=\exp(D)$ the following statement  holds
  $$\mathbb{Q}[x_0,x_1,\ldots,x_n]^{\varphi}=\ker D,$$
 see \cite{Now}, Proposition 6.1.4. 
  For  the integer polynomial transformations we may introduce an analogue of above notations.
  \begin{oz} A transformation   $D({\bf F}(\mathcal{A})):=(D(f_n(x_0,x_1,\ldots,x_m)|_{(a_0,a_1,\ldots,a_m)})_{n \geq 0},$ 
is called the $D$-de\-ri\-va\-ti\-ve of polynomial trans\-for\-mation ${\bf F}=(f_n(a_0,a_1,\ldots,a_m) )_{n \geq 0},$ $f \in \mathbb{Z}[x_0,x_1,\ldots,x_m].$
\end{oz}

\begin{oz}
A  transformation \text{ ${\bf F}$}  is called exponential if there exists  a locally nilpotent  derivation $D$ such that 
$$
{\bf F}(\mathcal{A})=\exp D (\mathcal{A}).
$$
\end{oz}
We may rewrite now the Theorem 1  for an exponential transformation:
\begin{theorem}
Suppose  a transformation $\textbf{F}$ is exponential and ${\bf F}(\mathcal{A})=\exp D (\mathcal{A})$ for some localy nilpotent derivation $D.$  Then a polynomial transformation  $\textbf{G}$ is  ${\bf F}$-invariant if and only if   $D(\textbf{G}(\mathcal{A}))={\bf 0},$ ${\bf 0}=(0,0,0,\ldots ).$
\end{theorem}
\begin{proof}
Suppose that  the transformation $\textbf{G}$ is ${\bf F}$-invariant. Then by Theorem 1 we  have 
$$\textbf{G}(\mathcal{A})= \left(g_n(a_0,a_1,\ldots,a_m) \right )_{n \geq 0},$$ where $g_n(x_0,a_1,\ldots,x_m) \in \mathbb{Z}[x_0,x_1,\ldots,x_m]^{\varphi},$  for $\varphi=\exp D.$  Since the automorphism $\varphi$ is exponential, we have that $D(g_n(x_0,a_1,\ldots,x_m))=0.$ Thus $D(\textbf{G}(\mathcal{A}))=0.$

Suppose now  that the  transformation $\textbf{G}$  has  the form $$\textbf{G}(\mathcal{A})= \left(g_n(a_0,a_1,\ldots,a_m) \right )_{n \geq 0},$$ and $D(\textbf{G}(\mathcal{A}))=0.$  It implies  that 
$D(g_n(a_0,a_1,\ldots,a_m))=0$ for all $n.$ Then  we have that the polynomial $g_n(x_0,x_1,\ldots,x_m)$ belongs to $ \mathbb{Q}[x_0,x_1,\ldots,x_n]^{\varphi}$  where $\varphi=\exp D.$ By Theorem 1  the transformation $\textbf{G}$ is  ${\bf F}$-invariant.
\end{proof}

The Weitzenb\"ok derivations are related with   some special transformations by the following theorem:
\begin{theorem}
The transformation  $F(\mathcal{A}):=
\left (a_n+\sum\limits_{i=0}^{n-1} \alpha_i a_i \mid  \alpha_i \in \mathbb{Z} \right )_{n \geq 0}
$
is exponential and    $F(\mathcal{A})=\exp D (\mathcal{A}),$
where the derivation  $D$ is a  Weitzenb\"ok derivation defined by   
$$
D(f) =\sum_{i=1}^{\infty} \frac{(-1)^{i+1}}{i} E^i(f),
$$
and   $E=\varphi-\textbf{1}$ is a locally nilpotent map. 
\end{theorem}
The proof follows from \cite{Essen}, Proposition 2.1.3.

Thus  this yields an  algorithm for solve   Problem 1 in the  case when the transformation $\textbf{F}(\mathcal{A})=(b_n)_{n \geq 0}$ has  the special form 
$$
b_n=a_n+\sum\limits_{i=0}^{n-1} \alpha_i a_i, \alpha_i \in \mathbb{Z}.
$$
In this case for the corresponding polynomial automorphism  $\varphi(x_n)=x_n+\sum\limits_{i=0}^{n-1} \alpha_i x_i $ we find the explicit form of the Weitzenb\"ok derivation   $D$ such that  $\varphi=\exp(D)$ (Theorem 3). After that we find any $(\mathcal{D},D)$-intertwining automorphism $\Psi$ and obtain that $\ker D =\Psi \left( \ker \mathcal{D} \right). $ Then an arbitrary sequence of kernel elements defines  $\textbf{F}$-invariant   transformation  (Theorem 1.).

For  solving    Problem 2 for a transformation $\textbf{F}=(b_n=f_n(a_0,a_1,\ldots,a_m))_{n \geq 0}$ we find a locally nilpotent derivation $D$ of $\mathbb{Z}[x_0,a_1,\ldots,x_m]$ such that   $f_n(x_0,a_1,\ldots,x_m) \in \ker D.$ It can be done by the method of indefinite coefficients. So we define the automorphism   $\varphi=\exp D$ and the transformation $\textbf{G}(\mathcal{A})=(b_n=\varphi(a_n))_{n \geq 0}.$  Thus, the transformation  $\textbf{F}$ is  $\textbf{G}$-invariant by Theorem 1. 


\section{The $\mu$-binomial transformations.}

We     use the developed techniques  to get  another  proof of   the following well known result:
 
\begin{theorem}[\cite{Lay},\cite{SpS}]
The Hankel transformation $\textbf{H}$ is  $\textbf{B}_\mu$-invariant.
\end{theorem}
\begin{proof} We follows  the above algorithm. The corresponding to $\textbf{B}_\mu$ automorphism $\varphi_\mu$ has the form: $$\varphi_\mu(x_n)=\sum\limits_{i=0}^n {n \choose i} x_i \mu^{n-i}.$$
For the basic Weitzenb\"ok derivation $\mathcal{D}$  we have 
\begin{gather*}
\exp (\mu \mathcal{D})(x_n)=\sum_{i \geq 0} \frac{1}{i!} \left(\mu\mathcal{D}\right)^i(x_n)=\sum_{i=0}^{n} \frac{n(n{-}1)\ldots (n{-}(i{-}1))}{i!}\, \mu^i x_{n-i}=\\=\sum_{i=0}^{n} {n \choose i} \mu^i x_{n-i}=\sum\limits_{i=0}^n {n \choose i} x_i \mu^{n-i}.
\end{gather*}
Thus $\varphi_\mu=\exp (\mu \mathcal{D}).$ It follows that the transformation $\textbf{B}_\mu$  is exponential, i.e., $\textbf{B}_\mu=\exp (\mu \mathcal{D}) \mathcal{A}.$  Since the catalectiant belongs to the kernel of  derivation $\mathcal{D}$ we  have   that $\mathcal{D}(\textbf{H}(\mathcal{A}))=\textbf{0}$.  Then by    Theorem 3  we obtain  that  the transformation $\textbf{H}$ is  $\textbf{B}_\mu$-invariant. 

 The map $\exp (\mu \mathcal{D}):\mathbb{Q}[x_0,x_1,\ldots,x_n] \to \mathbb{Q}[x_0,x_1,\ldots,x_n]$ is a ring homomorphism, see \cite{Essen}, Proposition 1.2.24. It follows that $\varphi_{\mu_1+\mu_2}=\varphi_{\mu_1} \circ \varphi_{\mu_2}$. Therefore $\varphi_{\mu} \circ \varphi_{-\mu}$  is the identity map and  $\textbf{B}_\mu^{-1}=\textbf{B}_{-\mu}$. It follows  immediately that the inverse transformation $\textbf{B}_\mu^{-1}$ is also $\textbf{H}$-invariant transformation.\end{proof}

 All  $\textbf{B}_\mu$-invariant transformation form a group, see French \cite{French}.   The identity $\varphi_{\mu_1+\mu_2}=\varphi_{\mu_1} \circ \varphi_{\mu_2}$  implies that the  group $(\mathbb{Z},+)$ is a  subgroup  of those group.

 The following theorem is a  solution of  Problem 1   for  the $\mu$-binomial transformation:
 \begin{theorem} A   transformation  $\textbf{F}$ is   $\textbf{B}_\mu$-invariant  if and only if   ${\bf \mathcal{D}}\left(\textbf{F}(\mathcal{A})\right)={\bf 0}.$ 
 \end{theorem}
 The proof follows from Theorem 3.
 
 Theorem  3.2 \cite{B_App}   implies  the result 
 
 \begin{theorem}
 Let $\textbf{F}$  is arbitrary  $\textbf{B}_\mu$-invariant transformation. Then $\textbf{F}(\textbf{1})={\bf 0}$, where \\ {${\textbf{1}=(1,1,1,\ldots,1,\ldots).}$}
 \end{theorem}

 Below we offer  some of  Hankel-type transformations which arise from  the classical invariant theory. Note that all of those  transformations  are  $\textbf{B}_\mu$-invariant   and $\textbf{B}_\mu^{-1}$-invariant.

\subsection{Cayley transformation}

Put ${\rm \textbf{CAYLEY}}(\mathcal{A})=( b_{n+2})_{n \geq 0},$ 
$$
b_{n}=\sum_{i=0}^{n-2} (-1)^i {n \choose i} a_{n-i} a_1^i a_0^{n-k-1}+(n-1) (-1)^{n+1} a_1^n.
$$
The transformation is inspired by Theorem 2.
\subsection{Transvectant transformation}
Let  $\mathcal{A}=(a_n)_{n \geq 0},\mathcal{C}=(c_n)_{n \geq 0}$ be two sequences. The  trans\-for\-ma\-tion  ${\rm \textbf{TR}}(\mathcal{A},\mathcal{C})=(b_n)_{n \geq 0},$ where
$$
b_n=\sum_{i=0}^n (-1)^i {n \choose i} a_i c_{n-i},
$$
is called the transvectant transformation.
We have  $${\rm \textbf{Tr}}(\textbf{B}_\mu(\mathcal{A}),\textbf{B}_\mu(\mathcal{C}))={\rm \textbf{Tr}}(\mathcal{A},\mathcal{C}).$$

In the case  $\mathcal{C}=\mathcal{A}$ we  get 
$$
b_n=\sum_{i=0}^n (-1)^i {n \choose i} a_i a_{n-i}.
$$

\subsection{Resultant   transformation}
Let  $\mathcal{A}=(a_n)_{n \geq 0},\mathcal{C}=(c_n)_{n \geq 0}$ be two sequences. The transformation  ${\rm \textbf{RES}}(\mathcal{A},\mathcal{C})=(b_n)_{n \geq 0}$  where $b_n$
is  the leading coefficient  of resultant of the polynomials 
$$
P_n(\mathcal{A})=\sum_{i=0}^{n} {n \choose i} a_i X^{n-i}, P_n(\mathcal{C})=\sum_{i=0}^{n} {n \choose i} c_i X^{n-i},
$$
is called the \textit{resultant transformation}.
\subsection{Discriminant   transformation}

The transformation  ${\rm \textbf{DISCR} }(\mathcal{A})=(b_n)_{n \geq 0}$  where $b_n$
is  the discriminant of the polynomial 
$$
P_{n+2}(\mathcal{A})=\frac{1}{(n+2)^{n+2}}\sum_{i=0}^{n+2}a_i {n+2 \choose i} X^{n+2-i},
$$
is called the \textit{discriminant transformation}.

\noindent
\textbf{Problem 3.} What is   the explicit form of $\Psi(\textbf{F})$ for $$\textbf{F} \in \{\textbf{CAYLEY, H, RES, DISCRIM, TR} \}?$$

\section{Examples.}

\subsection{ Transformation  ${\rm PSUM}(\mathcal{A})=(b_n=a_0+a_1+\ldots+a_n )_{n \geq 0}$} The corresponding locally nilpotent derivation  (see Theorem 3) has the form 
$$
D(x_n)=\sum_{i=1}^{\infty} \frac{(-1)^{i+1}}{i} E^i(x_n).
$$
We have 
\begin{align*}
&E(x_0)=0,E(x_n)=x_0+x_1+x_2+\cdots+x_{n-1},\\
&E^2(x_n)=\sum_{i=0}^{n-1}E(x_i)=\sum_{i=0}^{n-1}\sum_{j=0}^{i-1}x_j=\sum_{i=0}^{n-2} (n-1-i) x_i.
\end{align*}
By induction we obtain $E^i(x_n)=\displaystyle \sum\limits_{k=0}^{n-i} \displaystyle {n-i-1\choose i-1} x_k.$
Then \begin{gather*}
D(x_n)=\sum_{i=1}^{n} \frac{(-1)^{i+1}}{i} \sum_{k=0}^{n-i} \displaystyle {n-i-1\choose i-1} x_k=\sum_{k=0}^{n-1} \left( \sum_{i=0}^{n-1-k} \frac{(-1)^{i}}{i+1} { {n{-}1{-}k} \choose i}\right)x_k=\sum_{k=0}^{n-1} \frac{x_k}{n-k}.
\end{gather*}
Let us find  $(\mathcal{D},D)$-intertwining transformation $\Psi$. We   show that   
$$
\Psi(\mathcal{A})=\left \{\Psi(x_n)=\sum_{k=0}^n (-1)^{n+k} k! \left\{\begin{matrix} n \\ k \end{matrix} \right\} x_k,  \Psi(x_0)=x_0 \right \},
$$
where $\left\{\begin{matrix} n \\ k \end{matrix} \right\}$ is   Stirling number of the second kind. In fact,
\begin{gather*}
D\left(\Psi(x_n)\right)=D\left(\sum_{k=0}^n (-1)^{n+k} k! \left\{\begin{matrix} n \\ k \end{matrix} \right\} x_k\right)=\sum_{k=0}^n (-1)^{n+k} k! \left\{\begin{matrix} n \\ k \end{matrix} \right\} \sum_{i=0}^{k-1} \frac{x_i}{k-i}\,\,=\\
=\sum_{i=0}^{n-1} \sum_{j=i+1}^n (-1)^{n+j} \left\{\begin{matrix} n \\ j \end{matrix}\right\} \frac{j!}{j-i}\,\, x_i=n \sum_{i=0}^{n-1} (-1)^{n-1+i} \left\{\begin{matrix} n-1 \\ i \end{matrix}\right\}i!\, x_i =\Psi(\mathcal{D}(x_n)).
\end{gather*}
Therefore now  we  may construct a {\rm  PSUM}-invariant transformation  by using already  known $\textbf{B}_\mu $-invariant transformations and this $(\mathcal{D},D)$-intertwining transformation $\Psi$. For instance,    the transformation 
\begin{gather*}
\Psi(\textbf{H}(\mathcal{A}))=\{a_{{0}},-{a_{{1}}}^{2}-a_{{1}}a_{{0}}+2\,a_{{2}}a_{{0}},-4\,a_{{1}}a_{{2}}a_{{0}}+24\,a_{{1}}a_{{2}}a_{{3}}+24\,a_{{0}}a_{{1}}a_{{3}}+48\,a
_{{0}}a_{{2}}a_{{4}}-8\,{a_{{2}}}^{3}-\\-8\,a_{{0}}{a_{{2}}}^{2}-12\,a_{{
1}}{a_{{2}}}^{2}-36\,a_{{0}}{a_{{3}}}^{2}-4\,{a_{{1}}}^{2}a_{{2}}-24\,
{a_{{1}}}^{2}a_{{4}}+24\,{a_{{1}}}^{2}a_{{3}}-24\,a_{{0}}a_{{1}}a_{{4}
}, \ldots  \},
\end{gather*}
is {\rm  PSUM}-invariant.

\subsection{The Transformation ${\rm SUM}(\mathcal{A})=(b_n=a_n+a_{n-1} )_{n \geq 0}$}

We  have $\varphi(x_n)=x_n+x_{n-1},$  $E(x_n)=\varphi(x_n)-x_n=x_{n-1}$ and 
$$
D(x_n)=\sum_{i \geq 1} \frac{(-1)^{i+1}}{i} E^i(x_n)=\sum_{i=1}^n \frac{(-1)^{i+1}}{i} \, x_{n-i}.
$$ 
  Let  
$$
\Psi(x_0)=x_0, \Psi(x_n)=c_{n,1}x_1+c_{n,2} x_2 +\cdots+c_{n,n} x_n.
$$
The $(\mathcal{D},D)$-intertwining map satisfies the conditions   $D (\Psi(x_n))=\Psi\left(\mathcal{D}(x_n)\right)$. After a routine calculation we get that $c_{n,i}=i! \left\{\begin{matrix} n \\ i \end{matrix} \right \}$ and $(\mathcal{D},D)$-intertwining map has the form
$$
\Psi(x_n)=\sum_{i=1}^n i! \left\{\begin{matrix} n \\ i \end{matrix} \right\} x_i.
$$
Thus, the transformation 
$$
\Psi(\textbf{H}(\mathcal{A}))=\begin{vmatrix} \Psi(a_0) &  \Psi(a_1) & \Psi(a_2) &\cdots  & \Psi(a_n)  \\
\Psi(a_1) &  \Psi(a_2) & \Psi(a_3) &\cdots  & \Psi(a_{n+1})  \\
\hdotsfor{5}\\
 \Psi(a_{n-1}) &  \Psi(a_{n}) & \Psi(a_{n-1}) &\cdots  & \Psi(a_{2n-1})\\
 \Psi(a_n) &  \Psi(a_{n+1}) & \Psi(a_{n+2}) &\cdots  & \Psi(a_{2n})  
\end{vmatrix}.
$$
 is ${\rm SUM}$-invariant.

\subsection{ Transformation  ${\rm DIFF}(\mathcal{A}) =(b_n=a_n-a_{n-1} )_{n \geq 0}$} The corresponding automorphism has the form  $\varphi(x_n)=x_n-x_{n-1}.$ It  implies that $E(x_n)=-x_{n-1}$ and  $E^i(x_n)=(-1)^{i} x_{n-i}.$
Then the derivation $D$ is defined by 
$$
D(x_n)=\sum_{i \geq 1} \frac{(-1)^{i+1}}{i} E^i(x_n)=\sum_{i=1}^n \frac{(-1)^{i+1}}{i} \, (-1)^i x_{n-i}=-\sum_{i=1}^{n}\frac{x_{n-i}}{i}.
$$ 
The $( \mathcal{D},D)$-intertwining map  has  the form
$$
\Psi(x_n)=\sum_{i=1}^n (-1)^i i! \left\{\begin{matrix} n \\ i \end{matrix} \right\} x_i.
$$

\subsection{The  transformation $\textbf{F}=( b_n=\sum\limits_{i=0}^{2n} (-1)^i a_i a_{2n\!{-}i})_{n \geq 0}$.}

Let us try to solve   Problem 2  for this transformation.  To do it we  have  to find a suitable locally nilpotent derivation that satisfies the conditions   
$$D\left(\sum\limits_{i=0}^{2n} (-1)^i x_i x_{2n\!{-}i}\right)=0, n=0,1,\ldots .$$

Let us  consider the following locally nilpotent derivation $D(x_i)=x_{i-1}.$ In the author's paper \cite{A-Cas}  has been  proved  that for  the derivation $D$ holds $D\left(\sum\limits_{i=0}^{2n} (-1)^i x_i x_{2n\!{-}i}\right)=0.$ Let us calculate the exponential automorphism $\varphi=\exp D.$  We  have 
$$
\varphi(x_n)=D^0(x_n)+D(x_n)+\frac{1}{2!} D^2(x_n) +\cdots=x_n+x_{n-1}+\frac{1}{2!} \, x_{n-2}+\frac{1}{n!}\, x_0.
$$
Define a (rational!) transformation   by $\textbf{G}(\mathcal{A}):=( \varphi(x_n) )_{n \geq 0}$. Then $\textbf{F}(\textbf{G}(\mathcal{A}))=\textbf{F}(\mathcal{A}).$


\vspace{0.75 cm}
\textit{2000 Mathematics Subject Classification}: Primary 11B75; Secondary 13A50.

\noindent
\textit{Keywords:} invariant; polynomial transformation; Hankel determinant.

\end{document}